\documentclass[10pt,reqno]{amsart}
\usepackage{graphicx}
     \makeatletter
     \def\section{\@startsection{section}{1}%
     \z@{.7\linespacing\@plus\linespacing}{.5\linespacing}%
     {\bfseries%\normalfont\scshape
     \centering
     }}
     \def\@secnumfont{\bfseries}
     \makeatother
\newtheorem{theorem}{Theorem}[section]
\newtheorem{lemma}[theorem]{Lemma}

\newtheorem{corollary}[theorem]{Corollary}
\theoremstyle{definition}
\newtheorem{definition}[theorem]{Definition}
\newtheorem{example}[theorem]{Example}
\theoremstyle{remark}
\newtheorem{remark}[theorem]{Remark}
\numberwithin{equation}{section}
\begin{document}

\title[Action functionals for SDE with L\'evy noise]{Action functionals for stochastic differential equations with L\'evy noise}

\author{Shenglan Yuan}
\address{Shenglan Yuan: Center for Mathematical Sciences, Huazhong University of Sciences and Technology, Wuhan, 430074, China}
\email{shenglanyuan@hust.edu.cn}

\author{Jinqiao Duan}
\address{Jinqiao Duan: Department of Applied Mathematics, Illinois Institute of Technology, Chicago, IL 60616, USA}
\email{duan@iit.edu}

\subjclass[2010] {Primary 60F10; Secondary 65C30}

\keywords{Action functionals; Large deviation; Stochastic differential equations; L\'evy noise}

\begin{abstract}
By using large deviation theory that deals with
the decay of probabilities of rare events on an exponential scale, we study the longtime behaviors and establish action functionals for scaled Brownian motion and L\'evy processes with existing finite exponential moments. Based on extended contraction principle, Legendre transform and L\'evy symbols,
we derive the action functionals for stochastic differential equations driven by L\'evy processes.
\end{abstract}

\maketitle

\section{Introduction}
 Stochastic effects are ubiquitous in complex systems of science and engineering \cite{Ar13}. Although random mechanisms may appear to be very small or very fast, their long time impacts on the system evolution may be delicate or even profound \cite{Ku97}. Mathematical modeling of complex systems under uncertainty often leads to stochastic differential equations (SDEs); see \cite{BSW14,Ku04,Pr04,Ro05}. Fluctuations appeared in the SDEs are often non-Gaussian rather than Gaussian.

In the last years, the long time large deviations behavior of slow-fast systems has attracted more and more attention because of the various applications in the fields of statistical physics, engineering, chemistry and financial mathematics \cite{BCD13,FK06,Ki04}. Large deviations for SDEs driven by Brownian motion and random meaures are proved in the monographs \cite{DZ10,G08,LP92}. Large deviation results for SDEs with L\'evy noise are obtained by \cite{D18,K14}.

Action functionals play an important role for large deviation theory \cite{FW98,L10}.
The main goal of this paper is to derive the action funtionals for SDEs driven by L\'evy processes of the form
\begin{equation*}
dX_{t}^{\varepsilon}=b(X_{t-}^{\varepsilon})dt+\sqrt{\varepsilon}\sigma(X_{t-}^{\varepsilon})dB_{t}+\eta(X_{t-}^{\varepsilon})dL_{t}^{\varepsilon},
\end{equation*}
where $L_{t}^{\varepsilon}:=\varepsilon L_{\frac{t}{\varepsilon}}$ is a scaled L\'evy process with finite exponential moments.

We first show that the scaled L\'evy process satisfies a large deviation principle, and obtain its action functional.
Then we construct continuous mappings to get an exponentially good approximations .
Finally, we derive the action functionals for solutions of SDEs with L\'evy noise by using extended contraction principle, Legendre transform and L\'evy symbols. To keep notation as simple as possible, we restrict ourselves to one-dimensional processes. Most of the results can be proved in a similar fashion as for multi-dimensional processes.

 In Section 2, we recall some basic concepts, and introduce extensions of the contraction principle.
 In Section 3, we focus on action functionals for scaled Brownian motion, and obtain the action functionals for solutions of SDEs driven by Brownian motion (Lemma \ref{Bro} and Theorem \ref{hai}).
 In Section 4, we establish action functionals for scaled L\'evy processes, and consider the action functionals for solutions of SDEs with L\'evy noise  (Lemma \ref{iBro}, Theorem \ref{lan} and Corollary \ref{tai}).

\section{Prelimilaries}
%that possess a square integrable derivative, equipped with the norm
%:=\{\int_{0}^{t}f(s)ds:f\in L^{2}[0,1]\}
%\begin{equation*}
%||f||:=[\int_{0}^{1}|f'(t)|^{2}dt]^{\frac{1}{2}}.
%\end{equation*}
Throughout, $(\Omega,\mathcal{F},\mathbb{P})$ is a probability space. We consider Euclidean space $\mathbb{R}$
endowed with the Borel $\sigma$-algebra $\mathcal{B}(\mathbb{R})$. Let $\lfloor x\rfloor$ denote the integer part of $x\in\mathbb{R}$. As usual, $\mathcal{C}[0,1]$ is the space of all continuous functions $f: [0,1]\rightarrow\mathbb{R}$ such that $f(0)=0$,
equipped with the uniform norm
\begin{equation*}
\|f\|_{\infty}:=\sup_{t\in[0,1]}|f(t)|.
\end{equation*}
We denote by $\mathcal{D}[0,1]$ the space of real-valued c\'adl\'ag (right continuous with finite left limits) functions on $[0,1]$ endowed with
the supremum norm topology, and the $\sigma$-algebra $\mathcal{B}:=\sigma(\pi_{t}; t\in[0,1])$ generated by the projections $\pi_{t}: \mathcal{D}[0,1]\rightarrow\mathbb{R}$, $f\mapsto f(t), t\in[0,1]$. Note that $\mathcal{B}$ equals the Borel $\sigma$-algebra generated by the $J_{1}$-metric. Let $AC[0,1]$ denotes the space of all absolutely continuous functions with value
$0$ at $0$. The notation $BV[0,1]$ is the space of functions with bounded variation.\\

To transform a large deviation principle under a continuous mapping, we introduce the following contraction principle and investigate its extensions. They will be a crucial tool for studying action functionals of SDEs with L\'evy noise.

\begin{theorem}(\,\cite[Theorem 4.2.1]{DZ10}\,)
Let $(M_{1},d_{1})$, $(M_{2},d_{2})$ be metric spaces and $f: M_{1}\rightarrow M_{2}$ be a continuous function. Suppose that a family $(\mu_{\varepsilon})_{\varepsilon>0}$ of probability measures on $M_{1}$ satisfies a large deviation principle with action functional $I$.
Then the sequence of image measures $(\nu_{\varepsilon})_{\varepsilon>0}$ defined by $\nu_{\varepsilon}:=\mu_{\varepsilon}\circ f^{-1}$ on $M_{2}$, obeys a large deviation principle with action functional
\begin{equation*}
S(y):=\inf\{I(x): x\in M_{1}, y=f(x)\}.
\end{equation*}
\end{theorem}
\begin{proof}
Since $I$ is lower semicontinuous, it attains its minimum on compact sets. This implies that for any $y\in M_{2}$ and $S(y)<\infty$,
there exists $x\in M_{1}$ such that $f(x)=y$ and $S(y)=I(x)$. Then
\begin{equation*}
\Phi_{S}(r)=\{y\in M_{2}; S(y)\leq r\}=f(\Phi_{I}(r))~~~~\text{for}~~~~r\geq0.
\end{equation*}
In particular, $\Phi_{S}(r)$ is compact, i.e., $S$ is an action functional. Now let $U$ be an open set in $M_{1}$. Since $f$ is continuous, we know
$f^{-1}(U)$ is open. Apply the large deviation lower bound to $f^{-1}(U)$  and obtain
\begin{equation*}
\liminf_{\varepsilon\rightarrow0}\varepsilon\log\nu_{\varepsilon}(U)=\liminf_{\varepsilon\rightarrow0}\varepsilon\log\mu_{\varepsilon}(f^{-1}(U))\geq-\inf_{x\in f^{-1}(U)}I(x)=-\inf_{y\in U}S(y).
\end{equation*}
When $F$ is a closed set in $M_{1}$, the upper bound
\begin{equation*}
\limsup_{\varepsilon\rightarrow0}\varepsilon\log\nu_{\varepsilon}(F)=\limsup_{\varepsilon\rightarrow0}\varepsilon\log\mu_{\varepsilon}(f^{-1}(F))\leq-\inf_{x\in f^{-1}(F)}I(x)=-\inf_{y\in F}S(y)
\end{equation*}
follows in the same way.
\end{proof}

The following theorem provides a relation between large deviation principles of exponentially good approximations.

\begin{theorem}\label{ega}
Let $(X_{\varepsilon,m})_{\varepsilon>0,m\in\mathbb{N}}$ be an exponentially good approximation of $(X_{\varepsilon})_{\varepsilon>0}$ on a metric space $(M,d)$, i.e,
\begin{equation}\label{ie}
\lim_{m\rightarrow\infty}\limsup_{\varepsilon\rightarrow0}\varepsilon\log\mathbb{P}(d(X_{\varepsilon,m},X_{\varepsilon})>\delta)=-\infty,\,\,\,\,\text{for all}\,\,\,\,\delta>0,
\end{equation}
 such that
$X_{\varepsilon,m}$ satisfies a large deviation principle with action functional $S_{m}$ as $\varepsilon\rightarrow0$.\\
(i) $(X_{\varepsilon})_{\varepsilon>0}$ satisfies a weak large deviation principle with action functional
\begin{equation}\label{srf}
S(x):=\sup_{\delta>0}\liminf_{m\rightarrow\infty}\inf_{y\in B(x,\delta)}S_{m}(y).
\end{equation}
(ii) If $S$ is an action functional and
\begin{equation}\label{rf}
\inf_{x\in F}S(x)\leq\sup_{\delta>0}\limsup_{m\rightarrow\infty}\inf_{x\in F}S_{m}(x).
\end{equation}
holds for each closed set $F\subseteq M$, then $(X_{\varepsilon})_{\varepsilon>0}$ satisfies a large deviation principle with action functional $S$.
\end{theorem}
\begin{proof}
(i) In order to prove \eqref{srf}, it suffices to show that for any $x\in M$,
\begin{equation}\label{show}
S(x)=-\inf_{\delta>0}\limsup_{\varepsilon\rightarrow0}\varepsilon\log\mathbb{P}(X_{\varepsilon}\in B[x,\delta])=-\inf_{\delta>0}\liminf_{\varepsilon\rightarrow0}\varepsilon\log\mathbb{P}(X_{\varepsilon}\in B(x,\delta)).
\end{equation}
 Fix $\delta>0$ and $x\in M$. From
\begin{equation*}
\mathbb{P}(X_{\varepsilon,m}\in B(x,\delta))\leq\mathbb{P}(X_{\varepsilon}\in B(x,2\delta))+\mathbb{P}(d(X_{\varepsilon,m},X_{\varepsilon})>\delta),
\end{equation*}
we find, by the large deviation lower bound for $(X_{\varepsilon,m})_{\varepsilon>0}$,
\begin{align*}
-\inf_{y\in B(x,\delta)}S_{m}(y)&\leq\liminf_{\varepsilon\rightarrow0}\varepsilon\log\mathbb{P}(X_{\varepsilon,m}\in B(x,\delta))\\
                                &\leq\max\{\liminf_{\varepsilon\rightarrow0}\varepsilon\log\mathbb{P}(X_{\varepsilon}\in B(x,2\delta)),\limsup_{\varepsilon\rightarrow0}\varepsilon\log\mathbb{P}(d(X_{\varepsilon,m},X_{\varepsilon})>\delta)\}.
\end{align*}
Since $(X_{\varepsilon,m})_{\varepsilon>0,m\in\mathbb{N}}$ is an exponentially good approximation,
\begin{equation}\label{get}
\inf_{\delta>0}\liminf_{\varepsilon\rightarrow0}\varepsilon\log\mathbb{P}(X_{\varepsilon}\in B(x,2\delta))\geq\inf_{\delta>0}\limsup_{m\rightarrow\infty}(-\inf_{y\in B(x,\delta)}S_{m}(y))=-S(x).
\end{equation}
By interchanging the roles of $X_{\varepsilon,m}$ and $X_{\varepsilon}$, we get
\begin{multline*}
\limsup_{\varepsilon\rightarrow0}\varepsilon\log\mathbb{P}(X_{\varepsilon}\in B[x,\delta])\\
~~~~~~~\leq\max\{\liminf_{\varepsilon\rightarrow0}\varepsilon\log\mathbb{P}(X_{\varepsilon}\in B[x,2\delta]),\limsup_{\varepsilon\rightarrow0}\varepsilon\log\mathbb{P}(d(X_{\varepsilon,m},X_{\varepsilon})>\delta)\}.
\end{multline*}
Therefore, by the large deviation upper bound for $(X_{\varepsilon,m})_{\varepsilon>0}$ and \eqref{ie},
we obtain
\begin{equation}\label{we}
\inf_{\delta>0}\limsup_{\varepsilon\rightarrow0}\varepsilon\log\mathbb{P}(X_{\varepsilon}\in B[x,\delta])\leq\inf_{\delta>0}\limsup_{m\rightarrow\infty}(-\inf_{y\in B[x,2\delta]}S_{m}(y))=-S(x).
\end{equation}
Combining \eqref{get} and \eqref{we} yields \eqref{show}.

(ii) From the first part of this theorem that $(X_{\varepsilon})_{\varepsilon>0}$ satisfies a weak large deviation principle,
it remains to show the large deviation upper bound for any closed set $F\subseteq M$. Fix $\delta>0$, the large deviation upper bound for
$(X_{\varepsilon,m})_{\varepsilon>0}$ implies
\begin{align*}
  &\limsup_{\varepsilon\rightarrow0}\varepsilon\log\mathbb{P}(X_{\varepsilon}\in F)\\
  &\leq\max\{\limsup_{\varepsilon\rightarrow0}\varepsilon\log\mathbb{P}(X_{\varepsilon,m}\in F+B[0,\delta]),\limsup_{\varepsilon\rightarrow0}\varepsilon\log\mathbb{P}(d(X_{\varepsilon,m},X_{\varepsilon})>\delta)\}\\
  &\leq\max\{-\inf_{x\in F+B[0,\delta]}S_{m}(x),\limsup_{\varepsilon\rightarrow0}\varepsilon\log\mathbb{P}(d(X_{\varepsilon,m},X_{\varepsilon})>\delta)\}.
\end{align*}
Consequently, by \eqref{ie}
and \eqref{rf},
\begin{align*}
\limsup_{\varepsilon\rightarrow0}\varepsilon\log\mathbb{P}(X_{\varepsilon}\in F)&\leq-\lim_{\delta\rightarrow0}\limsup_{m\rightarrow\infty}\inf_{x\in F+B[0,\delta]}S_{m}(x)\\
                                                                                &\leq-\lim_{\delta\rightarrow0}\inf_{x\in F+B[0,\delta]}S(x)=-\inf_{x\in F}S(x).
\end{align*}
This finishes the proof.
\end{proof}

Now we present the extended contraction principle.
\begin{theorem}\label{ecp}
Let $(M_{1},d_{1})$, $(M_{2},d_{2})$ be metric spaces and $(X_{\varepsilon})_{\varepsilon >0}$ denotes a family of random variables obeying a large deviation principle in $(M_{1},d_{1})$ with action functional $I$. For $m\in\mathbb{N}$, let $f_{m}: M_{1}\rightarrow M_{2}$ be continuous functions and $f: M_{1}\rightarrow M_{2}$ measurable such that
\begin{equation}\label{di}
\limsup_{m\rightarrow\infty}\sup_{\{x:I(x)\leq r\}}d_{2}(f_{m}(x),f(x))=0~~~~~\text{for all}~~~~~{r\geq0}.
\end{equation}
Then for any family of random variables $(Y_{\varepsilon})_{\varepsilon >0}$ for which $(f_{m}(X_{\varepsilon}))_{\varepsilon>0, m\in\mathbb{N}}$ is an exponentially good approximation holds a large deviation principle with action functional
\begin{equation*}
S(y)=\inf\{I(x): y=f(x)\}.
\end{equation*}
\end{theorem}
 \begin{proof}
Since the functions $f_{m},\,m\in\mathbb{N}$, are continuous, the contraction principle entails that $(f_{m}(X_{\varepsilon}))_{\varepsilon>0}$
satisfies a large deviation principle with action functional
\begin{equation*}
S_{m}(y):=\inf\{I(x): y=f_{m}(x)\}.
\end{equation*}
Moreover, by \eqref{di}, $f$ is continuous on any sublevel set $\Phi_{I}(r):=\{x\in M_{1}: I(x)\leq r\}, r\geq0$. Hence, $S$ is an action functional with sublevel sets $f(\Phi_{I}(r))$. In view of Theorem \ref{ega}, it suffices to check \eqref{rf} and identify the action functional.

Fix $F\subseteq M_{2}$ closed and $\delta>0$, and also suppose
\begin{equation*}
c:=\liminf_{m\rightarrow\infty}\inf_{y\in F}S_{m}(y)<\infty.
\end{equation*}
Then we can choose a sequence $(x_{m})_{m\in\mathbb{N}}\subseteq M_{1}$ and $r>0$, such that $f_{m}(x_{m})\in F$ and $I(x_{m})=\inf_{y\in F}S_{m}(y)\leq r$. From \eqref{ecp} we have $f(x_{m})\in F+B[0,\delta]$ for $m=m(\delta)$ sufficiently large.
Thus,
\begin{equation*}
\inf_{y\in F+B(0,\delta)}S(y)\leq S(f(x_{m}))\leq I(x_{m})=\inf_{y\in F}S_{m}(y).
\end{equation*}
Taking $\delta\rightarrow0$ and $m\rightarrow\infty$, we infer
\begin{equation*}
\inf_{y\in F}S(y)\leq\liminf_{m\rightarrow\infty}\inf_{y\in F}S_{m}(y)=c.
\end{equation*}
Obviously, this inequality is trivially satisfied if $c=\infty$. In particular, \eqref{rf} holds. In order to identify the action functional, we use the preceding inequality for $F:=B[y,\delta]$ and let $\delta\rightarrow0$.
\end{proof}

\section{Brownian case}
Let $B_{t}$, $t\in[0,1]$ denotes a standard Brownian motion in $\mathbb{R}$. The logarithmic moment generating function of $B_{1}$ is
\begin{equation*}
\Lambda(\xi):=\log\mathbb{E}e^{\xi B_{1}}=\frac{1}{2}\xi^{2},
\end{equation*}
and the Legendre transform \cite{FK06} of $\Lambda$ is
\begin{equation*}
\Lambda^{*}(p):=\sup_{\xi\in\mathbb{R}}\{\xi p-\frac{1}{2}\xi^{2}\}=\sup_{\xi\in\mathbb{R}}\{-\frac{1}{2}(\xi-p)^{2}+\frac{1}{2}p^{2}\}=\frac{1}{2}p^{2}.
\end{equation*}

\begin{definition}
Let $\phi\in \mathcal{C}[0,1]$. The functional $S: \mathcal{C}[0,1]\rightarrow[0,\infty]$,
\begin{eqnarray}\label{rate}
S(\phi)=
\left\{\begin{array}{l}
   \frac{1}{2}\int_{0}^{1}|\phi'(t)|^{2}dt,\,\,\,\,\,\phi\in AC[0,1], \\
    \infty,\,\,\,\,\,\,\,\,\,\,\,\,\,\,\,\,\,\,\,\,\,\,\,\,\,\,\,\,\,\,\,\,\,\,\,\,\text{otherwise}
   \end{array}\right.
\end{eqnarray}
is the action functional of the Brownian motion $(B_{t})_{t\in[0,1]}$.
\end{definition}

\begin{lemma}\label{Bro}
The scaled Brownian motion $B_{t}^{\varepsilon}:=\varepsilon B_{\frac{t}{\varepsilon}}$ satisfies a large deviation principle in $(\mathcal{C}[0,1], \|\cdot\|_{\infty})$ as $\varepsilon\rightarrow0$
with action functional in \eqref{rate}, i.e.
\begin{align*}
  \liminf_{\varepsilon\rightarrow0}\varepsilon\log\mathbb{P}(\varepsilon B(\frac{.}{\varepsilon})\in U)&\geq-\inf_{\phi\in U}S(\phi), \\
  \limsup_{\varepsilon\rightarrow0}\varepsilon\log\mathbb{P}(\varepsilon B(\frac{.}{\varepsilon})\in F)&\leq-\inf_{\phi\in F}S(\phi),
\end{align*}
for any open set $U\subset \mathcal{C}[0,1]$ and closed set $F\subset\mathcal{ C}[0,1]$
\end{lemma}

\begin{proof}
In order to proof that $(B^{\varepsilon}_{t})_{t\in[0,1]}$ satisfies a large deviation
principle,  by the scaling property
\begin{equation*}
\varepsilon B(\frac{t}{\varepsilon})\stackrel{d}{=}\sqrt{\varepsilon}B(t),\,\,\,t\in[0,1],
\end{equation*}
where $``\stackrel{d}{=}"$ denotes equivalence (coincidence) in distribution, we may replace $B_{t}^{\varepsilon}$ by $\sqrt{\varepsilon}B_{t}$,  i.e.,
\begin{align*}
  \liminf_{\varepsilon\rightarrow0}\varepsilon\log\mathbb{P}(\sqrt{\varepsilon}B\in U)&\geq-\inf_{\phi\in U}S(\phi), \\
  \limsup_{\varepsilon\rightarrow0}\varepsilon\log\mathbb{P}(\sqrt{\varepsilon}B\in F)&\leq-\inf_{\phi\in F}S(\phi).
\end{align*}

For every $\phi_{0}\in U$, there is some $\delta_{0}>0$ such that
\begin{equation*}
\{\phi\in \mathcal{C}[0,1]: \|\phi-\phi_{0}\|_{\infty}<\delta_{0}\}\subset U.
\end{equation*}
Based on Schilder's theorem in \cite{S96},
\begin{equation*}
\mathbb{P}(\sqrt{\varepsilon}B\in U)\geq\mathbb{P}(\|\sqrt{\varepsilon}B-\phi_{0}\|_{\infty}<\delta_{0})\geq\exp[-\frac{1}{\varepsilon}(S(\phi_{0})+\gamma)],\,\,\text{for}\,\gamma>0.
\end{equation*}
Then
\begin{equation*}
\liminf_{\varepsilon\rightarrow0}\varepsilon\log\mathbb{P}(\sqrt{\varepsilon}B\in U)\geq S(\phi_{0}).
\end{equation*}
Since $\phi_{0}\in U$ is arbitrary,
\begin{equation*}
 \liminf_{\varepsilon\rightarrow0}\varepsilon\log\mathbb{P}(\sqrt{\varepsilon}B\in U)\geq-\inf_{\phi\in U}S(\phi).
\end{equation*}

Denote by
\begin{equation*}
\Phi(r):=\{f\in\mathcal{C}[0,1]: I(f)\leq r\},\,\,r\geq0
\end{equation*}
the sub-level sets of the action functional $S$ in \eqref{rate}. From Lemma 12.8 in \cite{SP12}, the action functional $S$ is lower semicontinuous. Then sub-level sets are closed. For each $r>0$, $0\leq s<t\leq1$ and $\phi\in\Phi(r)$, by Cauchy-Schwarz inequality,
\begin{align*}
  |\phi(t)-\phi(s)|&=|\int_{s}^{t}\phi'(u)du|\leq(\int_{s}^{t}|\phi'(u)|^{2}du)^{\frac{1}{2}}\sqrt{t-s}\\
                   &\leq\sqrt{2S(\phi)}\sqrt{t-s}\leq\sqrt{2r}\sqrt{t-s}.
\end{align*}
This implies that the family $\Phi(r)$ is equibounded and equicontinuous. Using Ascoli's theorem, $\Phi(r)$ is compact.
By the definition of the sub-level set $\Phi(r)$, we have $\Phi(r)\cap F=\emptyset$ for all $r<\inf_{\phi\in F}S(\phi)$.
So
\begin{equation*}
d(\Phi(r),F)=\inf_{\phi\in\Phi(r) }d(\phi,F)=:\delta_{r}>0.
\end{equation*}
Applying Schilder's theorem, we obtain that
\begin{equation*}
\mathbb{P}(\sqrt{\varepsilon}B\in F)\leq\mathbb{P}(d(\sqrt{\varepsilon}B,\Phi(r)>\delta_{r})\leq\exp(-\frac{r-\gamma}{\varepsilon}),\,\,\text{for}\,\gamma>0.
\end{equation*}
Hence
\begin{equation*}
\limsup_{\varepsilon\rightarrow0}\varepsilon\log\mathbb{P}(\sqrt{\varepsilon}B\in F)\leq -r.
\end{equation*}
Since $r<\inf_{\phi\in F}S(\phi)$ is arbitrary, we get
\begin{equation*}
\limsup_{\varepsilon\rightarrow0}\varepsilon\log\mathbb{P}(\sqrt{\varepsilon}B\in F)\leq-\inf_{\phi\in F}S(\phi).
\end{equation*}
\end{proof}

\begin{theorem}\label{hai}
Let $b, \sigma: \mathbb{R}\rightarrow\mathbb{R}$ be bounded, globally Lipschitz continuous functions such that $\inf_{x\in\mathbb{R}}\sigma(x)>0$, i.e., there exists $K>0$ such that
\begin{equation*}
|b(x)-b(y)|+|\sigma(x)-\sigma(y)|\leq K|x-y|,\,\,\,\,\text{for all}\,\,\,x, y\in\mathbb{R}.
\end{equation*}
Assume that $(X_{t}^{\varepsilon})_{t\in[0,1]}$ is a solution of the stochastic differential equation driven by Brownian motion, i.e. SDE of the form
\begin{equation*}
dX_{t}^{\varepsilon}=b(X_{t}^{\varepsilon})dt+\sqrt{\varepsilon}\sigma(X_{t}^{\varepsilon})dB_{t},~~~~X_{0}^{\varepsilon}=0.
\end{equation*}
Then $(X^{\varepsilon})_{\varepsilon>0}$ satisfies a large deviation principle in $(\mathcal{C}[0,1], \|\cdot\|_{\infty})$ with action functional
\begin{eqnarray}\label{aya}
S(\phi)=
\left\{\begin{array}{l}
\mathrm{}  \frac{1}{2}\int_{0}^{1}|\frac{\phi'(t)-b(\phi(t))}{\sigma(\phi(t))}|^{2}dt,\,\,\,\,\,\,\,\,\,\,\phi\in AC[0,1], \phi(0)=0,\\
    \infty,\,\,\,\,\,\,\,\,\,\,\,\,\,\,\,\,\,\,\,\,\,\,\,\,\,\,\,\,\,\,\,\,\,\,\,\,\,\,\,\,\,\,\,\,\,\,\,\,\,\,\,\,\,\,\,\,\,\,\,\,\,\,\,\,\,\,\,\,\,\,\text{otherwise}.
   \end{array}\right.
\end{eqnarray}
\end{theorem}

\begin{proof}
The key point of proof is that the family of solutions $(X_{t}^{\varepsilon,m})_{t\in[0,1]}$ given by the stochastic differential equation
\begin{equation}\label{eq}
dX_{t}^{\varepsilon,m}=b(X^{\varepsilon,m}_{\frac{\lfloor mt\rfloor}{m}})dt+\sqrt{\varepsilon}\sigma(X^{\varepsilon,m}_{\frac{\lfloor mt\rfloor}{m}})dB_{t}
\end{equation}
is an exponentially good approximation of $(X_{t}^{\varepsilon})_{t\in[0,1]}$. Then the stochastic integral can be evaluated pathwise.

Let $\delta, \rho, \varepsilon>0$. For $m\in\mathbb{N}$, define $F_{m}: \mathcal{C}[0,1]\longrightarrow \mathcal{C}[0,1]$ via $\phi=F^{m}(g)$, where
\begin{equation*}
\phi(t)=\phi(t_{k}^{m})+b(\phi(t_{k}^{m}))(t-t_{k}^{m})+\sigma(\phi(t_{k}^{m}))(g(t)-g(t_{k}^{m})),
\end{equation*}
for $t\in(t_{k}^{m},t_{k+1}^{m}]$, $t_{k}^{m}:=k/m$, $k=0,....,m-1$, and $\phi(0)=0$, such that $F^{m}(\sqrt{\varepsilon}B)=X^{\varepsilon,m}$. Define a $\mathcal{F}_{t}^{\varepsilon}$-stopping time
by
\begin{equation*}
\tau:=\tau(\rho):=\inf\{t\geq0: |X_{t}^{\varepsilon,m}-X^{\varepsilon,m}_{\frac{\lfloor mt\rfloor}{m}}|>\rho\}\wedge1,
\end{equation*}
and set
\begin{equation*}
b_{t}:=b(X^{\varepsilon,m}_{\frac{\lfloor mt\rfloor}{m}})-b(X_{t}^{\varepsilon}),\,\,\,\,\sigma_{t}:=\sigma(X^{\varepsilon,m}_{\frac{\lfloor mt\rfloor}{m}})-\sigma(X_{t}^{\varepsilon}),
\end{equation*}
where $\mathcal{F}_{t}^{\varepsilon}:=\sigma\{B_{s}^{\varepsilon}:s\leq t\}$  denotes the canonical filtration.
By the global Lipschitz continuity,
\begin{equation*}
|b_{t}|+|\sigma_{t}|\leq L|X^{\varepsilon,m}_{\frac{\lfloor mt\rfloor}{m}}-X_{t}^{\varepsilon}|\leq\sqrt{2}K(\rho^{2}+|X_{t}^{\varepsilon,m}-X_{t}^{\varepsilon}|^{2})^{\frac{1}{2}},\,\,\,\text{for any}\,\,\,t\in[0,\tau].
\end{equation*}
A calculation shows
\begin{equation*}
\varepsilon\log\mathbb{P}(\sup_{t\in[0,\tau]}|X_{t}^{\varepsilon,m}-X_{t}^{\varepsilon}|>\delta)\leq C+\log(\frac{\rho^{2}}{\rho^{2}+\delta^{2}}),
\end{equation*}
where $C>0$ is a constant that does not depend on $m,\varepsilon,\rho$. So,
\begin{equation*}
\lim_{\rho\rightarrow0}\sup_{m\geq1}\limsup_{\varepsilon\rightarrow0}\varepsilon\log\mathbb{P}(\sup_{t\in[0,\tau]}|X_{t}^{\varepsilon,m}-X_{t}^{\varepsilon}|>\delta)=-\infty,\,\,\,\,\text{for all}\,\,\,\,\delta>0.
\end{equation*}
Since $b$ and $\sigma$ are bounded, we find
\begin{equation*}
|X_{\frac{k}{m}+s}^{\varepsilon,m}-X_{\frac{k}{m}}^{\varepsilon,m}|\leq\tilde{C}(\frac{1}{m}+\sqrt{\varepsilon}\max_{0\leq k\leq m-1}\sup_{0\leq s\leq\frac{1}{m}}|B_{\frac{k}{m}+s}-B_{\frac{k}{m}}|),\,\,\,\,\text{for}\,\,\,\,0\leq s\leq\frac{1}{m},
\end{equation*}
where $\tilde{C}:=\max\{\|b\|_{\infty}, \|\sigma\|_{\infty}\}$. By the stationarity of the increments, we have
\begin{equation*}
\mathbb{P}(\tau<1)=\mathbb{P}(\bigcup_{k=0}^{m-1}\{\sup_{0\leq s\leq\frac{1}{m}}|X_{\frac{k}{m}+s}^{\varepsilon,m}-X_{\frac{k}{m}}^{\varepsilon,m}|>\rho\})\leq m\mathbb{P}(\sup_{0\leq s\leq\frac{1}{m}}|B_{s}|\geq\frac{\rho-\tilde{C}/m}{2\sqrt{\varepsilon}\tilde{C}}),
\end{equation*}
for all $m>\tilde{C}/\rho$. By Etemadi's inequality \cite{Et85} and Markov's inequality,
\begin{equation*}
\mathbb{P}(\sup_{0\leq s\leq\frac{1}{m}}|B_{s}|\geq\frac{\rho-\tilde{C}/m}{2\sqrt{\varepsilon}\tilde{C}})\leq6\exp(-\frac{\rho-\tilde{C}/m}{6\sqrt{\varepsilon}\tilde{C}}+\hat{C})
\end{equation*}
with a constant $\hat{C}>0$. Then
\begin{equation*}
\lim_{m\rightarrow\infty}\limsup_{\varepsilon\rightarrow0}\varepsilon\log\mathbb{P}(\tau<1)=-\infty\,\,\,\,\text{for all}\,\,\,\rho>0.
\end{equation*}
From
\begin{equation*}
\{\|X^{\varepsilon,m}-X^{\varepsilon}\|_{\infty}>\delta\}\subseteq\{\tau<1\}\cup\{\sup_{t\in[0,\tau]}|X_{t}^{\varepsilon,m}-X_{t}^{\varepsilon}|>\delta\},
\end{equation*}
the family of solutions $(X^{\varepsilon,m})_{\varepsilon>0,m\in\mathbb{N}}$ is indeed an exponentially good approximation of $(X^{\varepsilon})_{\varepsilon>0}$.

In Lemma \ref{Bro} we have shown that $B_{t}^{\varepsilon}$ obeys a large deviation principle with action functional $S$ as in \eqref{rate}, and $\sqrt{\varepsilon}B_{t}$ satisfies the same large deviation principle as $B_{t}^{\varepsilon}$. The task is now to find a function $F: \mathcal{C}[0,1]\rightarrow \mathcal{C}[0,1]$ such that the assumptions of Theorem \ref{ecp} are satisfied, and the continuous mappings $F^{m}$ converges uniformly on the compact sublevel set of $S$ in \eqref{rate} to $F$.

For absolutely continuous functions $g\in \mathcal{C}[0,1]$ and $x\in\mathbb{R}$, by $b$ and $\sigma$ are globally Lipschitz continuous, there exists
a unique solution $\phi=F(g)$ of the integral equation
\begin{equation*}
f(t)=\int_{0}^{t}b(f(s))ds+\int_{0}^{t}\sigma(f(s)g'(s))ds,\,\,\,\,t\in[0,1].
\end{equation*}
Fix $g\in\Phi(r):=\{\phi\in \mathcal{C}[0,1]: S(\phi)\leq r\}$. Using $b, \sigma$ are bounded and the Cauchy-Schwarz inequality,
\begin{equation*}
\sup_{0\leq t\leq1}|F^{m}(g)(t)-F^{m}(g)(\frac{\lfloor tm\rfloor}{m})|\leq\frac{\|b\|_{\infty}}{m}+\|\sigma\|_{\infty}\sqrt{\frac{1}{m}}\sqrt{\int_{0}^{1}g'(s)^{2}ds}=:\delta_{m}\stackrel{m\rightarrow0}{\longrightarrow}0.
\end{equation*}
Similarly,
\begin{align*}
  d(t):=&|F^{m}(g)(t)-F(g)(t)|\\
        &\leq K\int_{0}^{t}(1+|g'(s)|)|F^{m}(g)(\frac{\lfloor ms\rfloor}{m}-F(g)(s)|ds\\
        &\leq K(1+\sqrt{2r})\delta_{m}+L\int_{0}^{t}(1+|g'(s)|)d(s)ds.
\end{align*}
From Gronwall's lemma,
\begin{align*}
  d(t)&\leq K(1+\sqrt{2r})\delta_{m}[1+K\int_{0}^{t}(1+|g'(s)|)\exp(K\int_{s}^{t}(1+|g'(u)|)du)ds]\\
      &\leq K(1+\sqrt{2r})\delta_{m}(1+K(1+\sqrt{2r})e^{K(1+\sqrt{2r})}).
\end{align*}
Because the constants $K,1+\sqrt{2r},\delta_{m}$ do not depend on $t$ and $g$,
\begin{equation*}
\sup_{g\in\Phi(r)}||F^{m}(g)-F(g)||_{\infty}\leq K(1+\sqrt{2r})\delta_{m}(1+K(1+\sqrt{2r})e^{K(1+\sqrt{2r})})\stackrel{m\rightarrow0}{\longrightarrow}0.
\end{equation*}
Apply Theorem \ref{ecp}, $(X^{\varepsilon})_{\varepsilon>0}$ satisfies a large deviation principle with good rate function
\begin{eqnarray*}
S(\phi)=
\left\{\begin{array}{l}
\mathrm{}   \frac{1}{2}\int_{0}^{1}|g'(t)|^{2}dt,\,\,\,\,\,\,\,\,\,\,\,\,\,\,\,\,\,\,\,\phi\in AC[0,1], \phi(0)=0,\\
    \infty,\,\,\,\,\,\,\,\,\,\,\,\,\,\,\,\,\,\,\,\,\,\,\,\,\,\,\,\,\,\,\,\,\,\,\,\,\,\,\,\,\,\,\,\,\,\,\,\,\,\,\,\,\,\,\,\,\,\,\,\,\,\text{otherwise},
   \end{array}\right.
\end{eqnarray*}
where the infimum is taken over all functions $g\in AC[0,1]$, $g(0)=0$, such that
\begin{equation*}
\phi(t)=F(g)(t)=\int_{0}^{t}b(\phi(s))ds+\int_{0}^{t}\sigma(\phi(s))g'(s)ds,
\end{equation*}
that is,
\begin{equation*}
g'(t)=\frac{\phi'(t)-b(\phi(t))}{\sigma(\phi(t))}.
\end{equation*}
\end{proof}

\begin{remark}
For the special case $\sigma=1$, the solution  $(X^{\varepsilon})_{\varepsilon>0}$ of the stochastic differential equation
\begin{equation*}
dX_{t}^{\varepsilon}=b(X_{t}^{\varepsilon})dt+\sqrt{\varepsilon}dB_{t},~~~~X_{0}^{\varepsilon}=0,
\end{equation*}
has the action functional
\begin{eqnarray*}
S(\phi)=
\left\{\begin{array}{l}
\mathrm{}  \frac{1}{2}\int_{0}^{1}|\phi'(t)-b(\phi(t))|^{2}dt,\,\,\,\,\,\,\,\,\,\,\phi\in AC[0,1], \phi(0)=0,\\
    \infty,\,\,\,\,\,\,\,\,\,\,\,\,\,\,\,\,\,\,\,\,\,\,\,\,\,\,\,\,\,\,\,\,\,\,\,\,\,\,\,\,\,\,\,\,\,\,\,\,\,\,\,\,\,\,\,\,\,\,\,\,\,\,\,\,\,\,\text{otherwise}.
   \end{array}\right.
\end{eqnarray*}
\end{remark}

\begin{theorem}\label{nhai}
Let $b, \sigma: \mathbb{R}\rightarrow\mathbb{R}$ be bounded, globally Lipschitz continuous functions such that $\inf_{x\in\mathbb{R}}\sigma(x)>0$. There exists $K>0$ such that
\begin{equation*}
|b(x)-b(y)|+|\sigma(x)-\sigma(y)|\leq K|x-y|,\,\,\,\,\text{for all}\,\,\,x, y\in\mathbb{R}.
\end{equation*}
Then the action functional of solution $X_{t}$ for the SDE driven by Brownian motion
 \begin{equation}\label{no}
dX_{t}=b(X_{t})dt+\sigma(X_{t})dB_{t},\,\,X_{0}=0,
\end{equation}
is
 \begin{eqnarray*}
S(\phi):=
\left\{\begin{array}{l}
\mathrm{}  \int_{0}^{1}\frac{1}{2}|\frac{\phi'(t)-b(\phi(t))}{\sigma(\phi(t))}|^{2}dt,\,\,\,\,\,\,\,\,\,\,\,\,\,\,\phi\in AC[0,1], \phi(0)=0,\\
    \infty,\,\,\,\,\,\,\,\,\,\,\,\,\,\,\,\,\,\,\,\,\,\,\,\,\,\,\,\,\,\,\,\,\,\,\,\,\,\,\,\,\,\,\,\,\,\,\,\,\,\,\,\,\,\,\,\,\,\,\,\text{otherwise}.
   \end{array}\right.
\end{eqnarray*}
\end{theorem}

\begin{proof}
 The symbol of the solution $X_{t}$  for \eqref{no} is given by
 \begin{equation*}
q(x,\xi)=ib(x)\xi-\frac{1}{2}\sigma^{2}(x)\xi^{2}.
 \end{equation*}
Set
\begin{equation*}
H(x,\xi):=q(x,-i\xi)=b(x)\xi+\frac{1}{2}\sigma^{2}(x)\xi^{2}
\end{equation*}
then the Legendre transform of $H(x,\xi)$ is
\begin{align*}
L(x,\zeta)&=\sup_{\xi\in\mathbb{R}}[\zeta\xi-H(x,\xi)]\\
          &=\sup_{\xi\in\mathbb{R}}[\zeta\xi-b(x)\xi-\frac{1}{2}\sigma^{2}(x)\xi^{2}]\\
          &=\sup_{\xi\in\mathbb{R}}[-\frac{1}{2}\sigma^{2}(x)\big(\xi^{2}-\frac{2}{\sigma^{2}(x)}(\zeta-b(x))\xi\big)]\\
          &=\sup_{\xi\in\mathbb{R}}[-\frac{1}{2}\sigma^{2}(x)\big(\xi-\frac{\zeta-b(x)}{\sigma^{2}(x)}\big)^{2}+\frac{1}{2}|\frac{\zeta-b(x)}{\sigma^{2}(x)}|^{2}]\\
          &=\frac{1}{2}|\frac{\zeta-b(x)}{\sigma^{2}(x)}|^{2}.
\end{align*}
So the action functional of solution $X_{t}$ for \eqref{no} is
 \begin{eqnarray*}
S(\phi):=
\left\{\begin{array}{l}
\mathrm{}  \int_{0}^{1}L(\phi(t),\phi'(t))dt,\,\,\,\,\,\,\,\,\,\,\,\,\,\,\,\,\phi\in AC[0,1], \phi(0)=0,\\
    \infty,\,\,\,\,\,\,\,\,\,\,\,\,\,\,\,\,\,\,\,\,\,\,\,\,\,\,\,\,\,\,\,\,\,\,\,\,\,\,\,\,\,\,\,\,\,\,\,\,\,\,\,\,\,\,\,\,\,\,\,\text{otherwise},
   \end{array}\right.
\end{eqnarray*}
where
\begin{equation*}
L(\phi(t),\phi'(t))=\frac{1}{2}|\frac{\phi'(t)-b(\phi(t))}{\sigma(\phi(t))}|^{2}.
\end{equation*}
\end{proof}

\section{L\'evy case}
A stochastic process $L_{t}\in\mathbb{R}, t\in[0,1]$ is called a L\'evy process \cite[Chapter 7]{Du15} if the following properties hold:\\
(1) $L_{0}=0$(a.s.);\\
(2) $L$ has independent increments, i.e., for each $n\in\mathbb{N}$ and $0\leq t_{1}<t_{2}<...<t_{n+1}\leq1$, the random variables $(L_{t_{j+1}}-L_{t_{j}},1\leq j\leq n)$ are independent;\\
(3) $L$ has stationary increments, i.e., for each $n\in\mathbb{N}$ and $0\leq t_{1}<t_{2}<...<t_{n+1}\leq1$, $L_{t_{j+1}}-L_{t_{j}}\stackrel{\rm d}{=}L_{t_{j+1}-t_{j}}$;\\
(4) $L$ is stochastically continuous, i.e., $\lim_{t\downarrow0}\mathbb{P}(|L_{t}|>\varepsilon)=0$ for all $\varepsilon>0$;\\
(5) the paths $t\mapsto L_{t}$ are c\'adl\'ag with probability 1, that is, the trajectories are right continuous with existing left limits.

The L\'evy-It\^o decomposition \cite{S99} of L\'evy process $(L_{t})_{t\in[0,1]}$ with L\'evy triplet $(a,\sigma^{2},\nu)$ is
\begin{equation*}
L_{t}=at+\sigma B_{t}+\int_{0}^{t}\int_{|z|>1}zN(dz,ds)+\int_{0}^{t}\int_{0<|z|\leq1}z\tilde{N}(dz,ds),
\end{equation*}
where $(B_{t})_{t\in[0,1]}$ is a Brownian motion, $N$ denotes the jump counting measure, and $\tilde{N}$ is the compensated jump counting measure. The characteristic function of $(L_{t})_{t\in[0,1]}$
is given by the L\'evy-Khintchine formula \cite{Du15} :
 \begin{equation*}
 \mathbb{E}e^{i\xi L_{t}}=e^{t\psi(\xi)},\,\,\,\,\xi\in\mathbb{R},\,t\in[0,1],
 \end{equation*}
 where $\psi$ is the L\'evy symbol
 \begin{equation*}
\psi(\xi)=ia\xi-\frac{1}{2}\sigma^{2}\xi^{2}+\int_{\mathbb{R}\setminus\{0\}}(e^{i\xi y}-1-i\xi y\chi_{\{|y|\leq1\}})\nu(dy).
 \end{equation*}
There is a one-to-one correspondence between $\psi$ and $(a,\sigma^{2},\nu)$ consisting of the drift parameter $a\in\mathbb{R}$, the diffusion coefficient $\sigma\geq0$, and the L\'evy measure $\nu$ on $(\mathbb{R}\setminus\{0\},\mathcal{B}(\mathbb{R}\setminus\{0\}))$ satisfying $\int_{\mathbb{R}\setminus\{0\}}(y^{2}\wedge1)\nu(dy)<\infty$. Denote the logarithmic moment generating function of $L_{1}$ by
\begin{equation}\label{pu}
\Psi(\xi)=\psi(-i\xi)=a\xi+\frac{1}{2}\sigma^{2}\xi^{2}+\int_{\mathbb{R}\setminus\{0\}}(e^{\xi y}-1-\xi y\chi_{\{|y|\leq1\}})\nu(dy).
\end{equation}
If $\sigma=0$, we say that $(L_{t})_{t\in[0,1]}$ is a L\'evy process without
Gaussian component.

\begin{example}\label{eg}
The L\'evy measure of the a tempered stable L\'evy process $(L_{t})_{t\in[0,1]}$ is
\begin{equation*}
\nu(dy)=\frac{1}{2}\frac{\alpha(\alpha-1)}{\Gamma(2-\alpha)}e^{-my}\frac{dy}{|y|^{1+y}},\,\,\,\,\text{for}\,\alpha\in(1,2),\,m>0.
\end{equation*}
The L\'evy symbol of $(L_{t})_{t\geq0}$ is given by
\begin{equation*}
\psi(\xi)=-(|\xi|^{2}+m^{2})^{\frac{\alpha}{2}}\cos(\alpha\arctan\frac{|\xi|}{m})+m^{\alpha}.
\end{equation*}
\end{example}

\begin{definition}
Assume that L\'evy process $(L_{t})_{t\in[0,1]}$ with L\'evy triplet $(a,\sigma^{2},\nu)$ satisfies $\mathbb{E}e^{\lambda|L_{1}|}<\infty$, for all $\lambda\geq0$. The action functional of  $(L_{t})_{t\in[0,1]}$ on $(\mathcal{D}[0,1], \|\cdot\|_{\infty})$ is defined by
\begin{eqnarray}\label{lrate}
S(\phi)=
\left\{\begin{array}{l}
   \int_{0}^{1}\Psi^{*}(\phi'(t))dt,\,\,\,\,\,\,\,\phi\in AC[0,1],\phi(0)=0,\\
    \infty,\,\,\,\,\,\,\,\,\,\,\,\,\,\,\,\,\,\,\,\,\,\,\,\,\,\,\,\,\,\,\,\,\,\,\,\,\text{otherwise},
   \end{array}\right.
\end{eqnarray}
where $\Psi^{*}(\cdot)$ is the Legendre transform of $\Psi(\cdot)$ in \eqref{pu}.
\end{definition}

\begin{lemma}\label{iBro}
The scaled L\'evy process $L_{t}^{\varepsilon}:=\varepsilon L_{\frac{t}{\varepsilon}}, t\in[0,1]$ satisfies a large deviation principle in $(\mathcal{D}[0,1], \|\cdot\|_{\infty})$ as $\varepsilon\rightarrow0$ with action functional in \eqref{lrate}, i.e.,
\begin{align*}
  \liminf_{\varepsilon\rightarrow0}\varepsilon\log\mathbb{P}(L^{\varepsilon}\in U)&\geq-\inf_{\phi\in U}S(\phi), \\
  \limsup_{\varepsilon\rightarrow0}\varepsilon\log\mathbb{P}(L^{\varepsilon}\in F)&\leq-\inf_{\phi\in F}S(\phi),
\end{align*}
for any open set $U\in\mathcal{B}$ and closed set $F\in\mathcal{B}$.
\end{lemma}

\begin{proof}
In order to proof that $(L_{t}^{\varepsilon})_{t\in[0,1]}$ satisfies a large deviation
principle as $\varepsilon\rightarrow0$
with the action functional $S$ in \eqref{lrate}, we split the proof into several steps:

(i) The sequence of discretizations $(Z_{n}^{L})_{n\in\mathbb{N}}$ defined by
\begin{equation*}
\frac{Z_{n}^{L}(t,\omega)}{n}:=\frac{1}{n}L(\lfloor n.t\rfloor,\omega)=\frac{1}{n}[\sum_{j=0}^{n-1}L(j,\omega)\chi_{[\frac{j}{n},\frac{j+1}{n})}(t)+L(n,\omega)\chi_{\{1\}}(t)]
\end{equation*}
is exponentially tight in $(\mathcal{D}[0,1], \|\cdot\|_{\infty})$.

Since the mapping
\begin{equation*}
(\mathbb{R}^{n},|\cdot|)\ni x\mapsto(T_{n}x)(t):=\sum_{j=1}^{n-1}x_{j}\chi_{[\frac{j}{n},\frac{j+1}{n})}(t)+x_{n}\chi_{\{1\}}(t)\in(\mathcal{D}[0,1], \|\cdot\|_{\infty})
\end{equation*}
is continuous, we obtain that $T_{n}(K)$ is compact for any compact set $K\subseteq\mathbb{R}^{n}$. For $K\subseteq\mathbb{R}$ compact, we have
\begin{equation*}
\mathbb{P}(\frac{Z_{n}^{L}}{n}\notin T_{n}(K^{n}))\leq\sum_{j=1}^{n-1}\mathbb{P}(\frac{L_{j}}{n}\notin K).
\end{equation*}
The distribution of $\frac{L_{j}}{n}$ is a probability measure on $(\mathbb{R},\mathcal{B}(\mathbb{R}))$, hence $\frac{L_{j}}{n}$ is tight. For $j=1,...,n$, we conclude that $\frac{Z_{n}^{L}}{n}$ is tight in $(\mathcal{D}[0,1], \|\cdot\|_{\infty})$.
Fix $r>0$ and $\epsilon>0$, for $K\subseteq\mathbb{R}$ and $n\geq m$, we get
{\small \begin{align}\nonumber
\mathbb{P}(d(\frac{Z_{n}^{L}}{n},T_{m}(K^{m}))>\epsilon)&\leq\mathbb{P}(\frac{Z_{n}^{L}}{n}\notin T_{n}(K^{n}))+\mathbb{P}(\frac{Z_{n}^{L}}{n}\in T_{n}(K^{n}), d(\frac{Z_{n}^{L}}{n},T_{m}(K^{m}))>\epsilon) \\\label{have}
                                                        &=:I_{1}+I_{2}.
\end{align}}
We choose $K:=[-r,r]$ and estimate the terms separately. Applying Etemadi's inequality and Markov's inequality yields
\begin{equation*}
I_{1}=\mathbb{P}(\sup_{1\leq j\leq n}|\frac{L_{j}}{n}|>r)\leq3\sup_{1\leq j\leq n}\mathbb{P}(|L_{j}|>\frac{nr}{3})\leq3\sup_{1\leq j\leq n}\mathbb{E}e^{|L_{j}|-nr/3}\leq3e^{-nr/3}\beta_{1}^{n},
\end{equation*}
where $\beta_{1}:=\mathbb{E}e^{|L_{1}|}<\infty$ because $(L_{t})_{t\in[0,1]}$ has finite exponential moments. If
we set $f_{m}:=f(\frac{\lfloor m.\rfloor}{m})$, then
\begin{equation}\label{then}
d(f,T_{m}(K^{m}))\leq\|f-f_{m}\|_{\infty},\,\,\,\,\text{for all}\,\,\,\,f\in T_{n}(K^{n}).
\end{equation}
Moreover,
\begin{align}\nonumber
  \|f-f_{m}\|_{\infty}&=\max_{1\leq i\leq m-1}\sup_{t\in[\frac{i}{m},\frac{i+1}{m})}|f(t)-f_{m}(t)|\\ \nonumber
                      &=\max_{1\leq i\leq m-1}\sup_{t\in[\frac{i}{m},\frac{i+1}{m})}|f(\frac{\lfloor nt\rfloor}{n})-f(\frac{\lfloor mt\rfloor}{m})|\\ \label{over}
                      &\leq\max_{1\leq i\leq m-1}\sup_{1\leq j\leq\lfloor\frac{n}{m}\rfloor+1}|f(\frac{\lfloor n\frac{i}{m}\rfloor}{n}+\frac{j}{n})-f(\frac{\lfloor n\frac{i}{m}\rfloor}{n})|.
\end{align}
Combining \eqref{then} and \eqref{over},
\begin{align*}
  I_{2}&\leq\mathbb{P}(\sup_{1\leq i\leq m-1}\sup_{1\leq j\leq\lfloor\frac{n}{m}\rfloor+1}|Z_{n}^{L}(\frac{\lfloor n\frac{i}{m}\rfloor}{n}+\frac{j}{n})-Z_{n}^{L}(\frac{\lfloor n\frac{i}{m}\rfloor}{n})|>n\epsilon)\\
       &\leq\sum_{i=1}^{m-1}\mathbb{P}(\sup_{1\leq j\leq\lfloor\frac{n}{m}\rfloor+1}|L(\lfloor n\frac{i}{m}\rfloor+j)-L(\lfloor n\frac{i}{m}\rfloor)|>n\epsilon).
\end{align*}
By the stationarity and independence of the increments of $L$ and Markov's inequality,
\begin{align*}
I_{2}&\leq m\mathbb{P}(\sup_{1\leq j\leq\lfloor\frac{n}{m}\rfloor+1}|L_{j}|>n\epsilon)\leq3m\sup_{1\leq j\leq\lfloor\frac{n}{m}\rfloor+1}\mathbb{P}(|L_{j}|>\frac{n\epsilon}{3}) \\
     &\leq3m\sup_{1\leq j\leq\lfloor\frac{n}{m}\rfloor+1}\mathbb{E}e^{r|L_{j}|-nr\varepsilon/3}\leq3m\beta_{2}^{\lfloor\frac{n}{m}\rfloor+1}e^{-nr\varepsilon/3}, \\
\end{align*}
where $\beta_{2}:=\mathbb{E}e^{r|L_{1}|}<\infty$. Then
\begin{align*}
\limsup_{n\rightarrow\infty}\frac{1}{n}\log\mathbb{P}(d(\frac{Z_{n}^{L}}{n},T_{m}(K^{m}))>\epsilon)&\leq\max\{\log\beta_{1}-\frac{r}{3},\frac{1}{m}\log\beta_{2}-\frac{r\epsilon}{3}\} \\
  &\stackrel{r,m\rightarrow\infty}{\longrightarrow}-\infty.
\end{align*}
Consequently, $(Z_{n}^{L})_{n\in\mathbb{N}}$ is exponentially tight in $(\mathcal{D}[0,1], \|\cdot\|_{\infty})$.

(ii) $(Z_{n}^{L})_{n\in\mathbb{N}}$ satisfies a large deviation principle in $(\mathcal{D}[0,1], \|\cdot\|_{\infty})$ with respect to $\mathcal{B}$ as $n\rightarrow\infty$ with action functional
\begin{equation}\label{function}
I(\phi)=\sup_{\alpha\in BV[0,1]\cap \mathcal{D}[0,1]}(\int_{0}^{1}\phi d\alpha-\frac{1}{2}\int_{0}^{1}\Psi(\alpha(1)-\alpha(s))ds).
\end{equation}
where $\Psi(\cdot)$ as in \eqref{pu}.
Note that
\begin{equation*}
Z^{L}_{n}=\sum_{j=1}^{n-1}L_{j}\chi_{[\frac{j}{n},\frac{j+1}{n})}+L_{n}\chi_{\{1\}}=\sum_{j=1}^{n}(L_{j}-L_{j-1})\chi_{[\frac{j}{n},1]}.
\end{equation*}
By the stationarity and independence of the increments,
\begin{align*}
\mathbb{E}e^{\langle\alpha,Z^{L}_{n}\rangle}&=\mathbb{E}\exp(\sum_{j=1}^{n}(L_{j}-L_{j-1})(\alpha(1)-\alpha(\frac{j}{n})))\\
                                            &=\prod_{j=1}^{n}\mathbb{E}\exp(L_{1}(\alpha(1)-\alpha(\frac{j}{n}))).
\end{align*}
Since
\begin{equation*}
\mathbb{E}e^{\lambda L_{1}}=e^{\Psi(\lambda)},\,\,\,\,\text{for all}\,\,\,\,\lambda\in\mathbb{R},
\end{equation*}
we have
\begin{equation*}
\Lambda(\alpha):=\lim_{n\rightarrow\infty}\frac{1}{n}\log\mathbb{E}e^{\langle\alpha,Z^{L}_{n}\rangle}=\lim_{n\rightarrow\infty}\frac{1}{n}\sum_{j=1}^{n}\Psi(\alpha(1)-\alpha(\frac{j}{n}))=\int_{0}^{1}\Psi(\alpha(1)-\alpha(s))ds.
\end{equation*}
Pick $\beta\in BV[0,1]\cap \mathcal{D}[0,1]$ and set
\begin{equation*}
u(t,s):=\Psi((\alpha(1)-\alpha(s))+t(\beta(1)-\beta(s))),\,\,\,t\in[-1,1],\,s\in[0,1].
\end{equation*}
From $\alpha$, $\beta\in BV[0,1]$, it follows that $\|\alpha\|_{\infty}+\|\beta\|_{\infty}\leq C<\infty$. By $\mathbb{E}e^{\lambda|L_{1}|}<\infty$, for all $\lambda\geq0$,
we have
\begin{equation*}
-\infty<\log\mathbb{E}e^{-2C|L_{1}|}\leq|u(t,s)|\leq\log\mathbb{E}e^{2C|L_{1}|}<\infty.
\end{equation*}
And then
\begin{equation*}
|\partial_{t}u(t,s)|\leq2C\frac{1}{\mathbb{E}e^{-2C|L_{1}|}}\sqrt{\mathbb{E}(L_{1}^{2})}\sqrt{\mathbb{E}e^{2C|L_{1}|}}<\infty,\,\,\,\,\text{for all}\,\,\,t\in[-1,1].
\end{equation*}
We get
\begin{align*}
\frac{\Lambda(\alpha+t\beta)-\Lambda(\alpha)}{t}&\stackrel{t\rightarrow0}{\longrightarrow}\int_{0}^{1}\partial_{t}u(0,s)ds\\                                                                                         &=\int_{0}^{1}(\beta(1)-\beta(s))\frac{1}{\mathbb{E}e^{L_{1}(\alpha(1)-\alpha(s))}}\mathbb{E}(L_{1}e^{L_{1}(\alpha(1)-\alpha(s))})ds.
\end{align*}
So $\Lambda$ is $\mathcal{D}[0,1]$-G\^ateaux differentiable at $\alpha$, and its derivative equals
\begin{equation*}
D_{\alpha}(t):=\int_{0}^{t}\frac{1}{\mathbb{E}e^{L_{1}(\alpha(1)-\alpha(s))}}\mathbb{E}(L_{1}e^{L_{1}(\alpha(1)-\alpha(s))})ds,\,\,\,t\in[0,1].
\end{equation*}
We defer the rest proof of (ii) to (iv).

(iii) $(Z_{\lfloor\frac{1}{\varepsilon}\rfloor}^{L}/\lfloor\frac{1}{\varepsilon}\rfloor)_{\varepsilon>0}$ and $\varepsilon L(\frac{.}{\varepsilon})$ are exponentially equivalent.

Let $\epsilon>0$ and $r\geq0$. We have
\begin{equation}\label{O}
\|\frac{Z_{\lfloor\frac{1}{\varepsilon}\rfloor}^{L}}{\lfloor\frac{1}{\varepsilon}\rfloor}-\varepsilon L(\frac{.}{\varepsilon})\|_{\infty}\leq(\frac{1}{\lfloor\frac{1}{\varepsilon}\rfloor}-\varepsilon)\|Z_{\lfloor\frac{1}{\varepsilon}\rfloor}^{L}\|_{\infty}+\|\varepsilon Z_{\lfloor\frac{1}{\varepsilon}\rfloor}^{L}-\varepsilon L(\frac{.}{\varepsilon})\|_{\infty}:=C_{\varepsilon}+D_{\varepsilon}.
\end{equation}
We find
\begin{equation}\label{find}
\mathbb{P}(C_{\varepsilon}>\epsilon)\leq\mathbb{P}(\sup_{0\leq k\leq\lfloor\frac{1}{\varepsilon}\rfloor}|L_{k}|>\frac{1}{\varepsilon}(\frac{1}{\varepsilon}-1)\epsilon)\leq3\exp(-\frac{1}{\varepsilon}(\frac{1}{\varepsilon}-1)\frac{\epsilon}{3})\beta_{1}^{\lfloor\frac{1}{\varepsilon}\rfloor},
\end{equation}
where $\beta_{1}=\mathbb{E}e^{|L_{1}|}$ as in (i). Note that
\begin{equation*}
\sup_{t\in[0,1]}|Z_{\lfloor\frac{1}{\varepsilon}\rfloor}^{L}(t)-L(\frac{t}{\varepsilon})|\leq\sup_{0\leq k\leq\lfloor\frac{1}{\varepsilon}\rfloor}\sup_{l\in[0,2]}|L_{k+l}-L_{k}|\,\,\,\,\text{when}\,\,\,\,\frac{t}{\varepsilon}-\lfloor\lfloor\frac{1}{\varepsilon}\rfloor t\rfloor\leq2.
\end{equation*}
By Etemadi's inequality and the stationarity of the increments for $L$,
\begin{equation*}
\mathbb{P}(D_{\varepsilon}>\epsilon)\leq\mathbb{P}(\sup_{0\leq k\leq\lfloor\frac{1}{\varepsilon}\rfloor}\sup_{l\in[0,2]}|L_{k+l}-L_{k}|>\frac{\epsilon}{\varepsilon})\leq 3(\lfloor\frac{1}{\varepsilon}\rfloor+1)\sup_{l\in[0,2]}\mathbb{P}(|L_{l}|>\frac{\epsilon}{3\varepsilon}).
\end{equation*}
Since $(L_{t}-t\mathbb{E}L_{1})_{t\geq0}$ is a martingale, we know that $(e^{r|L_{t}-t\mathbb{E}L_{1}|})_{t\geq0}$ is a submartingale.
By Markov's inequality,
\begin{equation}\label{in}
\sup_{l\in[0,2]}\mathbb{P}(|L_{l}-l\mathbb{E}L_{1}|>\frac{\epsilon}{3\varepsilon})\leq e^{-r\epsilon/3\varepsilon}\mathbb{E}e^{r|L_{2}-2\mathbb{E}L_{1}|}=:\beta_{3}e^{-r\epsilon/3\varepsilon}.
\end{equation}
Combining \eqref{O}, \eqref{find} and \eqref{in} implies
\begin{multline*}
\limsup_{\varepsilon\rightarrow0}\varepsilon\log\mathbb{P}(\|\frac{Z_{\lfloor\frac{1}{\varepsilon}\rfloor}^{L}}{\lfloor\frac{1}{\varepsilon}\rfloor}-\varepsilon L(\frac{.}{\varepsilon})\|_{\infty}>2\epsilon)\\
\,\,\,\,\leq\max\{\limsup_{\varepsilon\rightarrow0}\varepsilon\log\mathbb{P}(C_{\varepsilon}>\epsilon),\limsup_{\varepsilon\rightarrow0}\varepsilon\log\mathbb{P}(D_{\varepsilon}>\epsilon)\}\leq-\frac{r\epsilon}{3}\stackrel{r\rightarrow\infty}{\longrightarrow}-\infty.
\end{multline*}
Hence $(Z_{\lfloor\frac{1}{\varepsilon}\rfloor}^{L}/\lfloor\frac{1}{\varepsilon}\rfloor)_{\varepsilon>0}$ and $\varepsilon L(\frac{.}{\varepsilon})$ are exponentially equivalent.

(iv) $(L_{t}^{\varepsilon})_{t\in[0,1]}$ satisfies a large deviation principle with action functional $I$  in \eqref{function} that
equals the action functional $S$ defined in \eqref{lrate}.

Fix $\varepsilon>0$, $0<s_{1}<t_{1}\leq......\leq s_{n}<t_{n}\leq1$, and $c=(c_{1},...,c_{n})\in\mathbb{R}^{n}$. We define
\begin{equation}\label{define}
\alpha(t):=\sum_{j=1}^{n}c_{j}\chi_{[s_{j},t_{j})}(t),\,\,\,\,t\in[0,1].
\end{equation}
Then $\alpha\in BV[0,1]\cap D[0,1]$ and
\begin{equation}\label{and}
\int_{0}^{1}\phi d\alpha=\sum_{j=1}^{n}c_{j}(\phi(s_{j})-\phi(t_{j})).
\end{equation}
Moreover,
\begin{align}\nonumber
\int_{0}^{1}\log\mathbb{E}e^{L_{1}(\alpha(1)-\alpha(s))}ds&=\sum_{j=1}^{n}\int_{0}^{1}\log\mathbb{E}e^{-c_{j}L_{1}}\chi_{[s_{j},t_{j})}(s)ds\\\label{more}
                                                          &\leq\log\mathbb{E}e^{||c||_{\infty}|L_{1}|}\sum_{j=1}^{n}(t_{j}-s_{j}).
\end{align}
By \eqref{function}, we obtain
\begin{equation*}
\int_{0}^{1}\phi d\alpha\leq I(\phi)+\int_{0}^{1}\log\mathbb{E}e^{L_{1}(\alpha(1)-\alpha(s))}ds.
\end{equation*}
Using \eqref{and} and \eqref{more}, we find
\begin{equation*}
\sum_{j=1}^{n}c_{j}(\phi(s_{j})-\phi(t_{j}))\leq I(\phi)+\log\mathbb{E}e^{||c||_{\infty}|L_{1}|}\sum_{j=1}^{n}(t_{j}-s_{j}).
\end{equation*}
In particular, for $c_{j}:=r{\rm sgn}(f(s_{j})-f(t_{j}))$ and $r>0$,
\begin{equation*}
\sum_{j=1}^{n}|\phi(t_{j})-\phi(s_{j})|\leq\frac{I(\phi)}{r}+\frac{\log\mathbb{E}e^{|L_{1}|r}}{r}\sum_{j=1}^{n}(t_{j}-s_{j}).
\end{equation*}
Choosing $r>0$ sufficiently large and $\delta>0$ sufficiently small, we see that
\begin{equation*}
\sum_{j=1}^{n}(t_{j}-s_{j})<\delta\Longrightarrow\sum_{j=1}^{n}|\phi(t_{j})-\phi(s_{j})|<\epsilon,
\end{equation*}
i.e., $\phi$ is absolutely continuous. Letting $t\rightarrow0$ and $r\rightarrow\infty$,
\begin{equation*}
|\phi(t)|\leq\frac{I(\phi)}{r}+t\frac{\log\mathbb{E}e^{|L_{1}|r}}{r}
\end{equation*}
yields $\phi(0)=0$. Hence $I(\phi)<\infty$ implies that $\phi$ is absolutely continuous and $\phi(0)=0$. Then there exists $f\in L^{1}[0,1]$
such that
\begin{equation*}
\phi(t)=\int_{0}^{t}f(s)ds,\,\,\,\,\,t\in[0,1].
\end{equation*}
So
\begin{align*}
\int_{0}^{1}\phi d\alpha-\int_{0}^{1}\Psi(\alpha(1)-\alpha(s))ds&=\int_{0}^{1}[f(s)(\alpha(1)-\alpha(s))-\Psi(\alpha(1)-\alpha(s))]ds\\
                                                                   &\leq\int_{0}^{1}\Psi^{*}(f(s))ds=\int_{0}^{1}\Psi^{*}(\phi'(s))ds=S(\phi),
\end{align*}
for any $\alpha\in BV[0,1]\cap D[0,1]$. Now we prove $I(\phi)\geq S(\phi)$ for $\phi\in AC[0,1]$, $\phi(0)=0$. By the monotone convergence theorem,
it suffices to show
\begin{equation*}
\int_{0}^{1}\Lambda_{k}(\phi'(s))ds\leq I(\phi)\,\,\,\text{where}\,\,\,\Lambda_{k}(x):=\sup_{|\alpha|\leq k}(\alpha x-\Psi(\alpha)), x\in\mathbb{R}, k\in\mathbb{N}.
\end{equation*}
Note that $\Lambda_{k}$ is convex and locally bounded, hence continuous.
From
\begin{equation*}
\sum_{j=0}^{n-1}\frac{\phi(\frac{j+1}{n})-\phi(\frac{j}{n})}{\frac{1}{n}}\chi_{[\frac{j}{n},\frac{j+1}{n})}(t)\longrightarrow \phi'(t)\,\,\,\,a.s.
\end{equation*}
and the dominated convergence theorem, we get
\begin{equation*}
\int_{0}^{1}\Lambda_{k}(\phi'(t))dt=\lim_{n\rightarrow\infty}\sum_{j=0}^{n-1}\frac{1}{n}\Lambda_{k}(n[\phi(\frac{j+1}{n})-\phi(\frac{j}{n})]).
\end{equation*}
As $\alpha\mapsto \alpha x-\Psi(\alpha)$ is continuous, we can choose $|\alpha(x)|\leq k$ such that
\begin{equation*}
\Lambda_{k}(x)=\alpha(x)x-\Psi(\alpha(x)).
\end{equation*}
For suitable $\alpha_{0}^{n},...,\alpha_{n-1}^{n}$,
\begin{align*}
\int_{0}^{1}\Lambda_{k}(\phi'(t))dt&=\lim_{n\rightarrow\infty}\sum_{j=0}^{n-1}[\alpha_{j}^{n}(\phi(\frac{j+1}{n})-\phi(\frac{j}{n}))-\frac{1}{n}\Psi(\alpha_{j}^{n})] \\
                                &=\lim_{n\rightarrow\infty}(\int_{0}^{1}\phi d\alpha^{n}-\int_{0}^{1}\Psi(\alpha^{n}(1)-\alpha^{n}(t))dt)\leq I(\phi),
\end{align*}
where $\alpha^{n}\in BV[0,1]\cap \mathcal{D}[0,1]$, $n\in\mathbb{N}$, is a step function of the form \eqref{define}.
Consequently, the action funtionals \eqref{lrate} and \eqref{function}, i.e., $S(\phi)=I(\phi)$.
\end{proof}

\begin{remark}
 Lemma \ref{Bro} does not apply to L\'evy processes with infinite moments of order $n$, for some $n\in\mathbb{N}$. In particular, $\alpha$-stable process with symbol $\psi(\xi)=|\xi|^{\alpha}$ is not covered because it has finite first order moment for $\alpha\in(1,2]$. But Lemma \ref{Bro} is valid for the tempered stable L\'evy process $(L_{t})_{t\in[0,1]}$ in Example \ref{eg}.
\end{remark}

\begin{theorem}\label{lan}
Let $b,\sigma, \eta: \mathbb{R}\rightarrow\mathbb{R}$ be bounded, Lipschitz continuous functions. There exists $K>0$ such that
\begin{equation*}
|b(x)-b(y)|+|\sigma(x)-\sigma(y)|+|\eta(x)-\eta(y)|\leq K|x-y|,\,\,\,\,\text{for all}\,\,\,\,x,y\in\mathbb{R}.
\end{equation*}
Let $(B_{t})_{t\geq0}$ be a Brownian motion and $(L_{t})_{t\geq0}$ be an independent L\'evy process with L\'evy triplet $(a,0,\nu)$
and symbol $\psi$ such that $\mathbb{E}e^{\lambda|L_{1}|}<\infty$, for all $\lambda\geq0$. Define a scaled L\'evy process as $L_{t}^{\varepsilon}:=\varepsilon L_{\frac{t}{\varepsilon}}$.
Then the family of solutions $(X^{\varepsilon})_{\varepsilon>0}$ of
\begin{equation}\label{tao}
dX_{t}^{\varepsilon}=b(X_{t-}^{\varepsilon})dt+\sqrt{\varepsilon}\sigma(X_{t-}^{\varepsilon})dB_{t}+\eta(X_{t-}^{\varepsilon})dL_{t}^{\varepsilon}
\end{equation}
satisfies a large deviation principle in $(\mathcal{D}[0,1],\|\cdot\|_{\infty})$ as $\varepsilon\rightarrow0$ with action funtional
\begin{eqnarray}\label{wani}
S(\phi):=
\left\{\begin{array}{l}
\mathrm{}  \inf\{\frac{1}{2}\int_{0}^{1}|g'(t)|^{2}dt+\int_{0}^{1}\Psi^{*}(h'(t))dt\},\,\,\,\,\,\phi\in AC[0,1], \phi(0)=0,\\
    \infty,\,\,\,\,\,\,\,\,\,\,\,\,\,\,\,\,\,\,\,\,\,\,\,\,\,\,\,\,\,\,\,\,\,\,\,\,\,\,\,\,\,\,\,\,\,\,\,\,\,\,\,\,\,\,\,\,\,\,\,\,\,\,\,\,\,\,\,\,\,\,\,\,\,\,\,\,\,\,\,\,\,\,\,\,\,\,\,\,\,\,\,\text{otherwise},
   \end{array}\right.
\end{eqnarray}
where $\Psi^{*}(\cdot)$ denotes the Legendre transform of $\Psi(\cdot)$ defined by
\begin{equation*}
\Psi(\xi):=\psi(-i\xi)=a\xi+\int_{\mathbb{R}\setminus\{0\}}(e^{\xi y}-1-\xi y\chi_{\{|y|\leq1\}})\nu(dy),\,\,\xi\in\mathbb{R},
\end{equation*}
and the infimum is taken over all functions $g,h\in AC[0,1]$, $g(0)=h(0)=0$, such that
\begin{equation*}
\phi(t)=F(g,h)(t)=\int_{0}^{t}b(\phi(s))ds+\int_{0}^{t}\sigma(\phi(s))g'(s)ds+\int_{0}^{t}\eta(\phi(s))h'(s)ds.
\end{equation*}
\end{theorem}
\begin{proof}
Since $(B_{t})_{t\geq0}$ and $(L_{t})_{t\geq0}$ are independent, $(B_{t},L_{t})_{t\geq0}$ is a L\'evy process.
Denote by $(Z_{\lfloor\frac{1}{\varepsilon}\rfloor}^{B}/\lfloor\frac{1}{\varepsilon}\rfloor)_{\varepsilon>0}$ and $(Z_{\lfloor\frac{1}{\varepsilon}\rfloor}^{L}/\lfloor\frac{1}{\varepsilon}\rfloor)_{\varepsilon>0}$ of the
approximations of $(B_{t}^{\varepsilon})_{t\in[0,1]}$ and $(L_{t}^{\varepsilon})_{t\in[0,1]}$. By straightforward modifications for the proof of Lemma \ref{Bro}, $(\sqrt{\varepsilon}B,L^{\varepsilon})_{\varepsilon>0}$ satisfies a large deviation principle in $\mathcal{D}[0,1]\times \mathcal{D}[0,1]$ endowed
with the norm
\begin{equation*}
\|(f_{1},f_{2})\|:=\|f_{1}\|_{\infty}+\|f_{2}\|_{\infty},\,\,\,\,f,g\in \mathcal{D}[0,1],
\end{equation*}
as $\varepsilon\rightarrow0$ with action funtional
\begin{eqnarray}\label{ani}
S(g,h)=
\left\{\begin{array}{l}
\mathrm{}  \frac{1}{2}\int_{0}^{1}|g'(t)|^{2}dt+\int_{0}^{1}\Psi^{*}(h'(t))dt,\,\,\,\,\,g,h\in AC[0,1], g(0)=h(0)=0,\\
    \infty,\,\,\,\,\,\,\,\,\,\,\,\,\,\,\,\,\,\,\,\,\,\,\,\,\,\,\,\,\,\,\,\,\,\,\,\,\,\,\,\,\,\,\,\,\,\,\,\,\,\,\,\,\,\,\,\,\,\,\,\,\,\,\,\,\,\,\,\,\,\,\,\,\,\,\,\,\,\text{otherwise}.
   \end{array}\right.
\end{eqnarray}

For $m\in\mathbb{N}$ define continuous mappings $F^{m}: \mathcal{D}[0,1]\times \mathcal{D}[0,1]\longrightarrow \mathcal{D}[0,1]$ via $\phi=F^{m}(g,h)$, where
\begin{equation*}
\phi(t)=\phi(t_{k}^{m})+b(\phi(t_{k}^{m}))(t-t_{k}^{m})+\sigma(\phi(t_{k}^{m}))(g(t)-g(t_{k}^{m}))+\eta(\phi(t_{k}^{m}))(h(t)-h(t_{k}^{m}))
\end{equation*}
 with $t\in(t_{k}^{m},t_{k+1}^{m}]$, $t_{k}^{m}:=\frac{k}{m}$, $k=0,....,m-1$, and $\phi(0):=\phi(0-)=0$. Using similar arguments as in the proof of Theorem \ref{hai}, the solutions $(X_{t}^{\varepsilon,m})_{m\in\mathbb{N}, \varepsilon>0}=(F^{m}(\sqrt{\varepsilon}B_{t},L_{t}^{\varepsilon}))_{m\in\mathbb{N}, \varepsilon>0}$ of the stochastic differential equation
\begin{equation*}
dX_{t}^{\varepsilon,m}=b(X^{\varepsilon,m}_{\frac{\lfloor mt\rfloor}{m}-})dt+\sqrt{\varepsilon}\sigma(X^{\varepsilon,m}_{\frac{\lfloor mt\rfloor}{m}-})dB_{t}+\eta(X^{\varepsilon,m}_{\frac{\lfloor mt\rfloor}{m}-})dL_{t}^{\varepsilon},\,\,\,\,X_{0}^{\varepsilon,m}=0,
\end{equation*}
are an exponential approximation of $(X^{\varepsilon})_{\varepsilon>0}$. For absolutely continuous functions $g, h\in \mathcal{D}[0,1]$, by $b$ and $\sigma$ are globally Lipschitz continuous, there exists
a unique solution $F(g,h)$ of the integral equation
\begin{equation*}
\phi(t)=\int_{0}^{t}b(\phi(s))ds+\int_{0}^{t}\sigma(\phi(s))g'(s)ds+\int_{0}^{t}\eta(\phi(s))h'(s)ds,\,\,\,\,t\in[0,1],
\end{equation*}
such that
\begin{equation*}
\lim_{m\rightarrow\infty}\sup_{(g,h)\in\Phi(r)}\|F^{m}(g,h)-F(g,h)\|_{\infty}=0,\,\,\,\,\text{for all}\,\,\,r\geq0,
\end{equation*}
where $\Phi(r):=\{(g,h)\in \mathcal{D}[0,1]\times \mathcal{D}[0,1]: S(g,h)\leq r\}$ is the sublevel set of the action functional $S$ defined in \eqref{ani}.
From Theorem \ref{ecp}, it follows that $(X^{\varepsilon})_{\varepsilon>0}$ satisfies a large deviation principle with action functional
$S$ as in \eqref{wani}.
\end{proof}

\begin{remark}
For the special case $\eta=0$, Theorem \ref{lan} concides with Theorem \ref{hai}.
\end{remark}

\begin{corollary}\label{tai}
 The symbol of the solution of the stochastic differential equation
 \begin{equation*}
dX_{t}=b(X_{t-})dt+\sigma(X_{t-})dB_{t}+\eta(X_{t-})dL_{t}
\end{equation*}
 is given by
\begin{align*}
q(x,\xi)&=ib(x)\xi-\frac{1}{2}\sigma^{2}(x)\xi^{2}+\psi(\eta(x)\xi)\\
        &=i(b(x)+a\eta(x))\xi-\frac{1}{2}\sigma^{2}(x)\xi^{2}+\int_{\mathbb{R}\setminus{0}}(e^{iy\eta(x)\xi}-1-iy\eta(x)\xi\chi_{\{|y|\leq1\}})\nu(dy).
\end{align*}
Set
\begin{equation*}
H(x,\xi):=q(x,-i\xi)=(b(x)+a\eta(x))\xi+\frac{1}{2}\sigma^{2}(x)\xi^{2}+\int_{\mathbb{R}\setminus{0}}(e^{y\eta(x)\xi}-1-y\eta(x)\xi\chi_{\{|y|\leq1\}})\nu(dy),
\end{equation*}
and denote the Legendre transform of $H(x,\xi)$ by
\begin{align*}
L(x,\zeta)&=\sup_{\xi\in\mathbb{R}}[\zeta\xi-H(x,\xi)]\\
          &=\sup_{\xi\in\mathbb{R}}[\zeta\xi-(b(x)+a\eta(x))\xi-\frac{1}{2}\sigma^{2}(x)\xi^{2}-\int_{\mathbb{R}\setminus{0}}(e^{y\eta(x)\xi}-1-y\eta(x)\xi\chi_{\{|y|\leq1\}})\nu(dy)].
\end{align*}

Suppose that $L(x,\zeta)$ satisfies\\
(H1) The function $(x,\zeta)\mapsto L(x,\zeta)$ is finite, i.e., for all $x,\zeta\in\mathbb{R}$, $L(x,\zeta)<\infty$. For any $r>0$, there exist constants $C_{1},C_{2}>0$ such that
\begin{equation*}
L(x,\zeta)+|\frac{\partial}{\partial\zeta}L(x,\zeta)|\leq C_{1}\,\,\,\,\,\text{and}\,\,\,\,\frac{\partial^{2}}{\partial\zeta^{2}}L(x,\zeta)> C_{2},\,\,\,\,\text{for all}\,\,\,\,x\in\mathbb{R}, |\zeta|\leq r.
\end{equation*}
(H2) Continuity condition:
\begin{equation*}
\sup_{|x-y|<\delta}\sup_{\zeta\in\mathbb{R}}\frac{L(x,\zeta)-L(y,\zeta)}{1+L(y,\zeta)}\stackrel{\delta\rightarrow0}{\longrightarrow}0.
\end{equation*}
Then the family of solutions $(X^{\varepsilon})_{\varepsilon>0}$ of \eqref{tao} has the action functional:
\begin{eqnarray}\label{xu}
S(\phi):=
\left\{\begin{array}{l}
\mathrm{}  \int_{0}^{1}L(\phi(t),\phi'(t))dt,\,\,\,\,\,\,\,\,\,\,\,\,\,\,\,\,\phi\in AC[0,1], \phi(0)=x,\\
    \infty,\,\,\,\,\,\,\,\,\,\,\,\,\,\,\,\,\,\,\,\,\,\,\,\,\,\,\,\,\,\,\,\,\,\,\,\,\,\,\,\,\,\,\,\,\,\,\,\,\,\,\,\,\,\,\,\,\,\,\,\text{otherwise},
   \end{array}\right.
\end{eqnarray}
where
\begin{align*}
L(\phi(t),\phi'(t))&=\sup_{\xi\in\mathbb{R}}[\phi'(t)\xi-(b(\phi(t))+a\eta(\phi(t)))\xi-\frac{1}{2}\sigma^{2}(\phi(t))\xi^{2} \\
                   &\,\,\,\,\,-\int_{\mathbb{R}\setminus{0}}(e^{y\eta(\phi(t))\xi}-1-y\eta(\phi(t))\xi\chi_{\{|y|\leq1\}})\nu(dy)].
\end{align*}
\end{corollary}

\begin{remark}
The action functionals \eqref{ani} and \eqref{xu}. The L\'evy symbol $\psi$ is twice differentiable because $(L_{t})_{t\geq0}$ has finite exponential moments. Then $\xi\mapsto H(x,\xi)$ is twice  differentiable and so is its Legendre transform $\zeta\mapsto L(x,\zeta)$.
\end{remark}

\begin{example}
Let $U: \mathbb{R}\rightarrow\mathbb{R}$ be a smooth enough function with a global point of minimum $0\in\mathbb{R}$. We consider the stochastic differential equation
\begin{equation*}
dX^{\varepsilon}_{t}=-\nabla U(X^{\varepsilon}_{t})dt+\varepsilon dL_{t}^{\varepsilon},\,\,\,X^{\varepsilon}_{0}=0.
\end{equation*}
Here, the scaled L\'evy process $L^{\varepsilon}$ is given by
\begin{equation*}
L_{t}^{\varepsilon}=\int_{0}^{t}\int_{\mathbb{R}}z\tilde{N}^{\frac{1}{\varepsilon}}(ds,dz),\,\, t\in[0,1],
\end{equation*}
where $\tilde{N}^{\frac{1}{\varepsilon}}$ is the compensated Poisson random measure with compensator $\frac{1}{\varepsilon}ds\bigotimes\nu$.
The intensity measure $\nu$ has the form
\begin{equation*}
\nu(dz)=e^{-|z|^{\alpha}}dz, \,\,\,\,\text{for some}\,\,\alpha>0.
\end{equation*}
The action functional of $(X^{\varepsilon}_{t})_{t\in[0,1]}$ is
\begin{eqnarray*}
S(\phi):=
\left\{\begin{array}{l}
\mathrm{}  \inf\{\int_{0}^{1}\int_{\mathbb{R}}(g(t,z)\ln g(t,z)-g(t,z)+1)\nu(dz)dt\},\,\,\,\phi\in AC[0,1], \phi(0)=0,\\
    \infty,\,\,\,\,\,\,\,\,\,\,\,\,\,\,\,\,\,\,\,\,\,\,\,\,\,\,\,\,\,\,\,\,\,\,\,\,\,\,\,\,\,\,\,\,\,\,\,\,\,\,\,\,\,\,\,\,\,\,\,\,\,\,\,\,\,\,\,\,\,\,\,\,\,\,\,\,\,\,\,\,\,\,\,\,\,\,\,\,\,\,\,\,\,\,\,\,\,\,\,\,\,\,\,\,\,\,\,\,\,\,\,\,\,\,\,\,\,\,\text{otherwise},
   \end{array}\right.
\end{eqnarray*}
such that
\begin{equation*}
\phi(t)=-\int_{0}^{t}\nabla U(\phi(s))ds+\int_{0}^{t}\int_{\mathbb{R}}z(g(s,z)-1)\nu(dz)ds.
\end{equation*}
\end{example}

\par\bigskip\noindent
{\bf Acknowledgment.} We would like to thank Wei Wei, Yong Chen, Franziska K\"uhn and Hina Zulfiqar for helpful discussions. This work was partly supported
by the National Natural Science Foundation of China (NSFC) Grant Nos. 11771161 and 11771449.

\bibliographystyle{amsplain}

\end{document}